\documentclass[11pt]{amsart}

\usepackage{amsfonts}
\usepackage{amsmath}

\newtheorem{thm}{Theorem}[section]

\newtheorem{prop}[thm]{Proposition}
\newtheorem*{prob*}{Problem}
\newtheorem*{thm*}{Theorem}

\theoremstyle{definition}
\newtheorem{defn}[thm]{Definition}

\newtheorem*{defn*}{Definition}
\newtheorem{rem}[thm]{Remark}

\newtheorem*{rem*}{Remark}
\numberwithin{equation}{section}

\newcommand{\bx}{\breve{x}}

\newcommand{\C}{\mathbb C}

\newcommand{\R}{\mathbb R}
\newcommand{\Y}{\mathbb Y}
\newcommand{\Z}{\mathbb Z}

\newcommand{\X}{\mathfrak{X}}

\DeclareMathOperator{\Conf}{Conf}

\DeclareMathOperator{\Pf}{Pf} 
 
\DeclareMathOperator{\Spectral}{Spectral}

\begin{document}
\title[Representation theory  and Pfaffian point processes]
 {\bf{Representation theory of the infinite symmetric group and Pfaffian point processes}}
\author{Eugene Strahov}

\thanks{Einstein Institute of Mathematics, The Hebrew University of
Jerusalem, Givat Ram, Jerusalem 91904. E-mail:
strahov@math.huji.ac.il. Supported in part by the US-Israel Binational Science
Foundation (BSF) Grant No.\ 2006333,
 and by the Israel Science Foundation (ISF) Grant No.\ 1441/08.\\ }


\commby{}
\begin{abstract}
We construct a family of Pfaffian point processes relevant for the harmonic analysis on the infinite symmetric group.
The correlation functions of these processes are representable as Pfaffians with matrix valued kernels.
We give explicit formulae for the matrix valued kernels in terms of the classical Whittaker functions. The obtained formulae have the same structure
as that arising in the study of symplectic ensembles of Random Matrix Theory.

The paper is an extended version of the author's talk at
Fall 2010 MSRI Random Matrix Theory program.
\end{abstract}
\maketitle
\section{Introduction}
Let $S(\infty)$ denote the group whose elements are finite permutations of $\{1,2,\ldots\}$.
The group $S(\infty)$ is called the infinite symmetric group, and it is a model example of a ``big''
group. The harmonic analysis for such groups  is an active
topic of modern research, with connections to different areas of
mathematics from enumerative combinatorics to random growth models
and to the theory of Painlev\'e equations.  A theory of harmonic
analysis on the infinite symmetric  and infinite-dimensional
unitary groups is developed  by Kerov, Olshanski and Vershik
\cite{KOV1, KOV2}, Borodin and Olshanski
\cite{Borodin-1,Borodin-2,Borodin-3,Borodin-4}.
 For an introduction to
harmonic analysis on the infinite symmetric group see Olshanski
\cite{olshanski2003}.  A paper by Borodin and Deift
\cite{BorodinDeift} studies differential equations arising in the
context of harmonic analysis on the infinite-dimensional unitary
group, and a paper by Borodin and Olshanski \cite{BorodinCombin}
describes the link to problems of enumerative combinatorics, and
to certain random growth models.

Borodin and Olshanski \cite{Borodin-1,Borodin-2,Borodin-3,Borodin-4} have shown that the problem of the harmonic analysis
on the infinite symmetric group
leads to  determinantal point processes,
which in many ways are similar to point processes associated with random matrix ensembles of the $\beta=2$ symmetry class.
On the other hand, in addition to ensembles of $\beta=2$ symmetry class, Random Matrix Theory deals with ensembles of $\beta=1$ and $\beta=4$ symmetry classes.
These ensembles (in contrast to those from $\beta=2$ symmetry class) lead to Pfaffian point processes.
In addition to random matrix problems, Pfaffian point processes appear in the theory of random partitions (see Rains \cite{rains}, Borodin
and Rains \cite{borodinrains}, Forrester, Nagao and Rains \cite{forresternagaorains} , Strahov \cite{strahov, strahov2}, Petrov \cite{petrov}), and in statistical mechanics (see, for example, Ferrari \cite{ferrari}).

In this paper we construct and investigate Pfaffian point processes relevant for the harmonic analysis on the infinite symmetric group.
In particular, we present explicit formulae for the  correlation functions, see Theorem \ref{THEOREMMAINRESULT}.

This paper is organized as follows.
Section 2 reviews the construction of a remarkable family of unitary representations $T_{z,\frac{1}{2}}$
associated with the infinite-dimensional  pair $(S(2\infty),H(\infty))$ ($(S(2\infty),H(\infty))$ is a Gelfand pair in the sense of Olshanski \cite{olshanskiGelfandPairs}).
Section 3 gives a spectral representation of the spherical functions of $T_{z,\frac{1}{2}}$, and introduces the spectral measures
$M_{z,\theta=\frac{1}{2}}^{\Spectral}$ (these objects are closely related with the $z$-measures with the Jack parameter $\theta=\frac{1}{2}$).
The spectral measures $M_{z,\theta=\frac{1}{2}}^{\Spectral}$ govern decomposition of $T_{z,\frac{1}{2}}$ into irreducible components,
and the irreducible components are parameterized by points of the Thoma set.

The problem of harmonic analysis under considerations is to describe $M_{z,\theta=\frac{1}{2}}^{\Spectral}$. This is done in
Section 4.  We use the idea proposed by Borodin and Olshanski
 and convert the points of the Thoma set into a point configuration. Then the spectral measure $M_{z,\theta=\frac{1}{2}}^{\Spectral}$ defines a random point processes which can be studied by standard tools of probability theory.
Theorem \ref{THEOREMMAINRESULT} is the main result of the present paper. It states that the point process defined by the (lifted) spectral measure
$\widetilde{M}_{z,\theta=\frac{1}{2}}^{\Spectral}$  is a Pfaffian point processes. Moreover, Theorem \ref{THEOREMMAINRESULT} gives an explicit formula
for the corresponding correlation functions in terms of the classical Whittaker functions.
It is remarkable that the correlation kernel of Theorem \ref{THEOREMMAINRESULT} has the same structure as the matrix Airy kernel arising in the  study of the  symplectic random matrix  ensembles.
This similarity suggests a possibility to study the obtained Pfaffian point process on the same level as Pfaffian point processes of Random Matrix Theory.

Finally, Section 5 contains the sketch of the proof of Theorem \ref{THEOREMMAINRESULT}.\\
\textbf{Acknowledgements.} I am   grateful to Alexei Borodin and Grigori Olshanski for
numerous discussions and many valuable comments at different
stages of this work. It is my pleasure to thank the organizers of the Fall 2010 MSRI Random Matrix Theory program
for the stimulating and encouraging environment they created at the program. 

\section{The representations $T_{z,\frac{1}{2}}$}
In this section we review the construction of family $T_{z,\frac{1}{2}}$ of unitary
representations of the group $S(2\infty)$. These representations
(introduced in Olshanski \cite{olshanskiletter}, Strahov
\cite{strahov1}) are parameterized by points $z\in
\C\setminus\{0\}$, and can be viewed as  analogues of the
generalized regular representations introduced in Kerov,
Olshanski, and Vershik \cite{KOV1, KOV2}.
\subsection{The spaces $ X(n)$ and their projective limit} Let
$S(2n)$ be the permutation group of $2n$ symbols realized as the group of permutations
of the set  $\{-n,\ldots,-1,1,\ldots,n\}$. Let $\breve{t}\in
S(2n)$ be the product of the transpositions $(-n,n),(-n+1,n-1),
\ldots, (-1,1)$. By definition, the group $H(n)$ is the
centralizer of $\breve{t}$ in $S(2n)$,
$$
H(n)=\left\{\sigma\biggl|\sigma\in S(2n), \sigma
\breve{t}\sigma^{-1}=\breve{t}\right\}.
$$
The group $H(n)$ is called the hyperoctahedral group of degree
$n$. It is known that $(S(2n),H(n))$ is a Gelfand pair, see
Macdonald \cite{macdonald}, VII, \S 2.

Set $X(n)=H(n)\setminus S(2n)$, so $X(n)$ is the space of right
cosets of the subgroup $H(n)$ in $S(2n)$. The set $X(n)$ can be
realized as the set of all pairings  of
$\{-n,\ldots,-1,1,\ldots,n\}$ into $n$ unordered pairs. Thus every
element $\breve{x}$ of $X(n)$ may be represented as a collection of
$n$ unordered pairs,
\begin{equation}\label{representationofelement}
\breve{x}\in X(n)\longleftrightarrow
\breve{x}=\biggl\{\{i_1,i_2\},\ldots,\{i_{2n-1},i_{2n}\}\biggr\},
\end{equation}
where $i_1,i_2,\ldots,i_{2n}$  are distinct elements of the set
$\{-n,\ldots,-1,1,\ldots,n\}$.

Given an element $\breve{x}'\in X(n+1)$ we define its derivative
element $\bx\in X(n)$ as follows. Represent $\breve{x}'$ as $n+1$
unordered pairs. If $n+1$ and $-n-1$ are in the same pair, then
$\breve{x}$ is obtained from $\breve{x}'$ by deleting this pair.
Suppose that $n+1$ and $-n-1$ are in different pairs. Then
$\breve{x}'$ can be written as
$$
\breve{x}'=\biggl\{\{i_1,i_2\},\ldots, \{i_m,-n-1\},\ldots,
\{i_k,n+1\},\ldots, \{i_{2n+1},i_{2n+2}\}\biggr\}.
$$
In this case $\breve{x}$ is obtained from $\breve{x}'$ by removing
$-n-1$, $n+1$  from pairs $\{i_m,-n-1\}$ and $\{i_k, n+1\}$
correspondingly, and by replacing these two  pairs, $\{i_m,-n-1\}$
and $\{i_k, n+1\}$, by one pair $\{i_m,i_k\}$.  The map
$\breve{x}'\rightarrow \breve{x}$, denoted by $p_{n,n+1}$, will be
referred to as the canonical projection of $X(n+1)$ onto $X(n)$.

Consider the sequence
$$
X(1)\leftarrow\ldots\leftarrow X(n)\leftarrow
X(n+1)\leftarrow\ldots
$$
of canonical projections, and let
$$
X=\varprojlim X(n)
$$
denote the projective limit of the sets $X(n)$. By definition, the
elements of $X$ are arbitrary sequences $\breve{x}=(\breve{x}_1,
\breve{x}_2,\ldots )$, such that $\breve{x}_n\in X(n)$, and
$p_{n,n+1}(\breve{x}_{n+1})=\breve{x}_n$. The set $X$ is a closed
subset of the compact space of all sequences $(\breve{x}_n)$,
therefore, it is a compact space itself.

In what follows we denote by $p_n$ the projection $X\rightarrow
X(n)$ defined by $p_n(\breve{x})=\breve{x}_n$.

Let $\bx$ be an element of $X(n)$. Then $\bx$ can be identified
with  arrow configurations on  circles. Such  arrow configurations
can be constructed as follows. Once $\bx$ is written as a
collection of $n$ unordered pairs, one can represent $\bx$ as a
union of cycles of the form
\begin{equation}\label{formcycle}
j_1\rightarrow -j_2\rightarrow j_2\rightarrow -j_3\rightarrow
j_3\rightarrow\ldots\rightarrow -j_k \rightarrow j_k \rightarrow
-j_1\rightarrow j_1,
\end{equation}
where $j_1,j_2,\ldots ,j_k$ are distinct integers from the set
$\{-n,\ldots ,n\}$.

 For example, take
\begin{equation}\label{element}
\breve{x}=\biggl\{\{1,3\},\{-2,5\}, \{2,-1\},
\{-3,-5\},\{4,-6\},\{-4,6\}\biggr\}.
\end{equation}
Then $\bx\in X(3)$, and it is possible to think about $\bx$ as a
union of two cycles, namely
\begin{equation}\label{circle1}
1\rightarrow 3\rightarrow-3\rightarrow-5\rightarrow
5\rightarrow-2\rightarrow 2\rightarrow-1\rightarrow 1, \nonumber
\end{equation}
and
\begin{equation}\label{circle2}
4\rightarrow-6\rightarrow 6\rightarrow -4\rightarrow 4. \nonumber
\end{equation}
Cycle  (\ref{formcycle}) can be represented as a circle with
attached arrows. Namely,  we put on a circle points labeled by
$|j_1|$, $|j_2|$,$\ldots$, $|j_k|$, and attach arrows to these
points according to the following rules. The arrow attached to
$|j_1|$ is directed clockwise. If the next integer in the cycle
(\ref{formcycle}), $j_2$, has the same sign as $j_1$, then the
direction of the arrow attached to $|j_2|$ is the same as the
direction of the arrow attached to $|j_1|$, i.e. clockwise.
Otherwise, if the sign of $j_2$ is opposite to the sign of $j_1$,
the direction of the arrow attached to $|j_2|$ is opposite to the
direction of the arrow attached to $|j_1|$, i.e. counterclockwise.
Next, if the integer $j_3$ has the same sign as $j_2$, then the
direction of the arrow attached to $|j_3|$ is the same as the
direction of the arrow attached to $|j_2|$, etc. For example,  the
representation  of the element $\bx$ defined by (\ref{element})
in terms of arrow configurations on circles is shown in Fig. 1.
\begin{figure}
\begin{picture}(100,150)
\put(-10,100){\circle{200}}
\put(100,100){\circle{200}}
\put(-10,120){\circle*{2}} \put(-10,80){\circle*{2}}
\put(-30,100){\circle*{2}} \put(10,100){\circle*{2}}
\put(-10,124){$1$} \put(-10,69){$5$} \put(-42,100){$2$}
\put(14,100){$3$}
\put(-10,120){\vector(1,0){10}} \put(-10,80){\vector(-1,0){10}}
\put(-30,100){\vector(0,1){10}} \put(10,100){\vector(0,1){10}}
\put(100,120){\circle*{2}} \put(100,80){\circle*{2}}
\put(100,120){\vector(1,0){10}} \put(100,80){\vector(-1,0){10}}
\put(100,124){$4$} \put(100,69){$6$}
\end{picture}
\caption{The representation of the element
$$
\breve{x}=\biggl\{\{1,3\},\{-2,5\}, \{2,-1\},
\{-3,-5\},\{4,-6\},\{-4,6\}\biggr\}
$$
in terms of  arrow configurations on circles. The first circle
(from the left) represents  cycle $1\rightarrow
3\rightarrow-3\rightarrow-5\rightarrow 5\rightarrow-2\rightarrow
2\rightarrow-1\rightarrow 1$, and the second circle represents
cycle $4\rightarrow-6\rightarrow 6\rightarrow -4\rightarrow 4$.}
\end{figure}
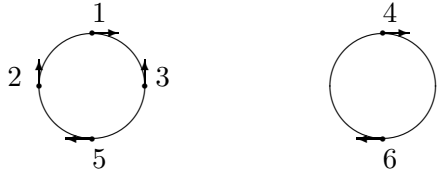
\subsection{The $t$-measures on $X$}
\begin{defn}\label{DefinitionEwensMeasures} For $t>0$ we set
$$
\mu_t^{(n)}(\breve{x})=\frac{t^{[\breve{x}]_n}}{t(t+2)\ldots
(t+2n-2)},
$$
where $\breve{x}\in X(n)$, and $[\breve{x}]_n$ denotes the number
of cycles in $\breve{x}$, or the number of circles in the
representation of $\bx$ in terms of arrow configurations.
\end{defn}
\begin{rem}
The measures $\mu_t^{(n)}$ on the spaces $X(n)$ are natural
analogues of the Ewens measures on the group $S(n)$ described in
Kerov, Olshanski and Vershik \cite{KOV2}.
\end{rem}
\begin{prop}\label{PROPOSITION4.2}
a) We have
\begin{equation}\label{normtequation}
\sum\limits_{\breve{x}\in X(n)}\mu_t^{(n)}(\breve{x})=1.
\end{equation}
Thus $\mu_t^{(n)}(\breve{x})$ can be understood as a probability
measure
on $X(n)$. \\
b) Given $t>0$, the canonical projections $p_{n,n+1}$ preserve the
measures $\mu_t^{(n)}(\breve{x})$, which means that the condition
\begin{equation}\label{mutnproperty}
\mu_t^{(n+1)}\biggl(\{\breve{x}'\;\vert\; \breve{x}'\in
X(n+1),p_{n,n+1}(\breve{x}')=\breve{x}\}\biggr)=\mu_t^{(n)}(\breve{x})
\end{equation}
is satisfied for each $ \breve{x}\in X(n)$.
\end{prop}
\begin{proof}
See Strahov \cite{strahov1}, Section 4.1
\end{proof}
It follows from Proposition \ref{PROPOSITION4.2} that for any
given $t>0$, the canonical projection $p_{n-1,n}$ preserves the
measures $\mu_t^{(n)}$. Hence the measure
$$
\mu_t=\varprojlim \mu_t^{(n)}
$$
on $X$ is correctly defined, and it is a probability measure.

Now we describe the right action of the group $S(2n)$ on the space
$X(n)$, and then we extend it to the right action of $S(2\infty)$
on $X$.

Let $\breve{x}_n\in X(n)$. Then $\breve{x}_n$ can be written as a
collection of $n$ unordered pairs (equation
(\ref{representationofelement})). Let  $g$ be a permutation from
$S(2n)$,
$$
g:\;\; \left(\begin{array}{ccccc}
  -n & -n+1 & \ldots & n-1 & n \\
  g(-n) & g(-n+1) & \ldots & g(n-1) & g(n) \\
\end{array}\right).
$$
The right action of the group $S(2n)$ on the space $X(n)$ is
defined by
$$
\breve{x}_n\cdot
g=\biggl\{\{g(i_1),g(i_2)\},\{g(i_3),g(i_4)\},\ldots,
\{g(i_{2n-1}),g(i_{2n})\}\biggr\}.
$$
\begin{prop}The canonical projection $p_{n,n+1}$ is equivariant
with respect to the right action of the group $S(2n)$ on the space
$X(n)$, which means
$$
p_{n,n+1}(\breve{x}\cdot g)=p_{n,n+1}(\breve{x})\cdot g,
$$
for all $\breve{x}\in X(n+1)$, and all $g\in S(2n)$.
\end{prop}
\begin{proof}
See Strahov \cite{strahov1}, Section 4.2
\end{proof}
Since the canonical projection $p_{n,n+1}$ is equivariant, the
right action of $S(2n)$ on $X(n)$ can be extended to the right
action of $S(2\infty)$ on $X$. For $n=1,2,\ldots $ we identify
$S(2n)$ with the subgroup of permutations $g\in S(2n+2)$
preserving the elements $-n-1$ and $n+1$ of the set
$\{-n-1,-n,\ldots,-1,1,\ldots,n,n+1\}$, i.e.
$$
S(2n)=\biggl\{g\biggl|g\in S(2n+2),\; g(-n-1)=-n-1,\; \mbox{and}\;
g(n+1)=n+1 \biggr\}.
$$

Let $S(2)\subset S(4)\subset S(6)\ldots $ be the collection of
such subgroups. Set
$$
S(2\infty)=\bigcup_{n=1}^{\infty}S(2n).
$$
Thus $S(2\infty)$ is the inductive limit of subgroups $S(2n)$,
$$
S(2\infty)=\varinjlim S(2n).
$$
If $\breve{x}=(\breve{x}_1,\breve{x}_2,\ldots )\in X$, and $g\in
S(2\infty)$, then the right action of $S(2\infty)$ on
$X=\varprojlim X(n)$,
$$X\times
S(2\infty)\longrightarrow X,
$$ is defined as $\breve{x}\cdot g=\check{y}$, where
$\breve{x}_n\cdot g=\breve{y}_n$ for all $n$ so large that $g\in
S(2\infty)$ lies in $S(2n)$.
\begin{prop}\label{Proposition1.444444}We have
$$
p_n(\breve{x}\cdot g)=p_n(\breve{x})\cdot g
$$
for all $\breve{x}\in X$, $g\in S(2\infty)$, and for all $n$ so
large that $g\in S(2n)$.
\end{prop}
\begin{proof}The claim follows immediately from the very
definition of the projection $p_n$, and of the right action of
$S(2\infty)$ on $X$.
\end{proof}
\begin{prop}
For any $\breve{x}=(\breve{x}_n)\in X$, and $g\in S(2\infty)$, the
quantity
$$
c(\breve{x};g)=[p_n(\breve{x}\cdot
g)]_n-[p_n(\breve{x})]_n=[p_n(\breve{x})\cdot g]_n-[p_n(\bx)]_n
$$
does not depend on $n$ provided that $n$ is so large that $g\in
S(2n)$.
\end{prop}
\begin{proof} See Strahov \cite{strahov1}, Section 4.3.
\end{proof}
\begin{prop} Each of measures $\mu_t$, $0<t<\infty$, is
quasiinvariant with respect to the action of $S(2\infty)$ on the
space $X=\varprojlim X(n)$. More precisely,
$$
\frac{\mu_t(d\bx\cdot g)}{\mu_t(d\bx)}=t^{c(\bx;g)};\;\;\bx\in
X,\; g\in S(2\infty),
$$
where $c(\bx;g)$ is the fundamental cocycle.
\end{prop}
\begin{proof}
See Strahov \cite{strahov1}, Section 4.4
\end{proof}

\subsection{Definition of $T_{z,\frac{1}{2}}$}
 Let $(\X,\Sigma,\mu)$ be
a measurable space. Let $G$ be a group which acts on $\X$ from the
right, and preserves the Borel structure. Assume that the measure
$\mu$ is quasiinvariant, i.e. the condition
$$
d\mu(\bx\cdot g)=\delta(\bx;g)d\mu(\bx)
$$
is satisfied for some nonnegative $\mu$-integrable function
$\delta(\bx;g)$ on $\X$, and for every $g$, $g\in G$. Set
\begin{equation}\label{3.1-3.1}
\left(T(g) f\right)(\bx)=\tau(\bx;g)f(\bx\cdot g),\; f\in
L^{2}(\X,\mu),
\end{equation}
where $|\tau(\bx;g)|^2=\delta(\bx;g)$. If
$$\tau(\bx;g_1g_2)=\tau(\bx\cdot
g_1;g_2)\tau(\bx;g_1),\; \bx\in\X, g_1,g_2\in G,
$$
then equation (\ref{3.1-3.1}) defines a unitary representation $T$
of $G$ acting in the Hilbert space $L^2(\X;\mu)$. The function
$\tau(\bx;g)$ is called a multiplicative cocycle.

Let $z\in\C$ be a nonzero complex number. We apply the general
construction described above for the space $\X=X$, the group
$G=S(2\infty)$, the measure $\mu=\mu_t$ (where $t=|z|^2$), and the
cocycle $\tau(\bx;g)=z^{c(\bx;g)}$. In this way we get a unitary
representation of $S(2\infty)$, $T_{z,\frac{1}{2}}$, acting in the
Hilbert space $L^2(X,\mu_t)$ according to the formula
$$
\left(T_{z,\frac{1}{2}}(g)f\right)(\bx)=z^{c(\bx;g)}f(\bx\cdot
g),\; f\in L^2(X,\mu_t),\;\bx\in X,\; g\in S(2\infty).
$$
This defines a family of unitary representations   of $S(2\infty)$.
\subsection{Spherical functions}
Let $(G,K)$ be a Gelfand pair, and let $T$ be a unitary
representation of $G$ acting in the Hilbert space $H(T)$. Assume
that $\xi$ is a unit vector in $H(T)$ such that $\xi$ is
$K$-invariant, and such that the span of vectors of the form
$T(g)\xi$ (where $g\in G$) is dense in $H(T)$. In this case $\xi$
is called the spherical vector, and the matrix coefficient
$(T(g)\xi,\xi)$ is called the spherical function of the
representation $T$. Two spherical representations are equivalent
if and only if their spherical functions coincide.

The representation $T_{z,\frac{1}{2}}$ is realized in the Hilbert
space $L^2(X,\mu_t)$, where $t=|z|^2$. Let $\mathbf{1}$ denote the
function on $X$ identically equal to $1$. It can be viewed as an
element of $L^2(X,\mu_t)$, and as a spherical vector of
$T_{z,\frac{1}{2}}$. Set
$$
\varphi_z(g)=\left(T_{z,\frac{1}{2}}(g)\mathbf{1},\mathbf{1}\right),\;\;g\in
S(2\infty).
$$
Thus $\varphi_z$ is the spherical function of $T_{z,\frac{1}{2}}$.

\section{Spectral representation of the spherical function  of
$T_{z,\frac{1}{2}}$}
\subsection{The $z$-measures on partitions with the general parameter $\theta>0$}
Denote the set of Young diagrams with
$n$ boxes   by $\Y_n$.
Let $M_{z,\bar z,\theta}^{(n)}$ be a probability measure on
$\Y_n$ defined by
\begin{equation}\label{EquationVer4zmeasuren}
M_{z,\bar z,\theta}^{(n)}=\frac{n!(z)_{\lambda,\theta}(\bar z)_{\lambda,\theta}}{(\frac{z\bar z}{\theta})_nH(\lambda,\theta)H'(\lambda,\theta)},
\end{equation}
where $n=1,2,\ldots $, and where we use the following notation
\begin{itemize}
    \item $z\in\C$ and $\theta>0$ are parameters.
    \item $(a)_n$ stands for the Pochhammer symbol,
    $$
    (a)_n=a(a+1)\ldots (a+n-1)=\frac{\Gamma(a+n)}{\Gamma(a)}.
    $$
    \item
    $(z)_{\lambda,\theta}$ is a multidemensional analogue of the
    Pochhammer symbol defined by
    $$
    (z)_{\lambda,\theta}=\prod\limits_{(i,j)\in\lambda}(z+(j-1)-(i-1)\theta)
    =\prod\limits_{i=1}^{l(\lambda)}(z-(i-1)\theta)_{\lambda_i}.
    $$
     Here $(i,j)\in\lambda$ stands for the box in the $i$th row
     and the $j$th column of the Young diagram $\lambda$, and we
     denote by $l(\lambda)$ the number of nonempty rows in the
     Young diagram $\lambda$.
    \item
    $$
    H(\lambda,\theta)=\prod\limits_{(i,j)\in\lambda}\left((\lambda_i-j)+(\lambda_j'-i)\theta+1\right),
   $$
   $$
     H'(\lambda,\theta)=\prod\limits_{(i,j)\in\lambda}\left((\lambda_i-j)+(\lambda_j'-i)\theta+\theta\right),
   $$
      where $\lambda'$ denotes the transposed diagram.
\end{itemize}

\begin{prop}\label{PropositionMSymmetries}
We have
$$
M_{z,\bar z,\theta}^{(n)}(\lambda)=M_{-z/\theta,-\bar z/\theta,1/\theta}^{(n)}(\lambda').
$$
\end{prop}
The probability measures $M_{z,\bar z,\theta}^{(n)}$ are called the $z$-measures
with an arbitrary Jack parameter $\theta>0$.
\subsection{The spectral $z$-measures with the general parameter $\theta>0$}
\begin{defn}
The space $\Omega$ of all pairs $\omega=(\alpha,\beta)$ of weakly
decreasing sequences of non-negative real numbers,
$$
\alpha=(\alpha_1\geq\alpha_2\geq\ldots\geq\alpha_k\geq\ldots\geq
0),\;\;\beta=(\beta_1\geq\beta_2\geq\ldots\geq\beta_k\geq\ldots\geq
0)
$$
such that
$$
\sum\limits_{k=1}^{\infty}\alpha_k+\sum\limits_{k=1}^{\infty}\beta_k\leq
1
$$
is called the Thoma set.
\end{defn}
Let us define an embedding of the algebra $\Lambda$ of symmetric
functions into the algebra of continuous functions on the Thoma
set $\Omega$. Since $\Lambda=\mathbb{C}[p_1,p_2,\ldots]$, where
$p_k=\sum_ix_i^k$ are power sums, it is sufficient to define the images
$\widetilde{p}_k$ of the $p_k$'s. We set
\begin{equation}\label{LambdaEmbedding}
\widetilde{p}_k(\omega|\theta)=\left\{%
\begin{array}{ll}
    \sum\limits_{j=1}^{\infty}\alpha_j^k+(-\theta)^{k-1}\sum\limits_{j=1}^{\infty}\beta_j^k, & k=2,3,\ldots, \\
    1, & k=1. \\
\end{array}%
\right.
\end{equation}
Then the following result holds true (see Borodin and Olshanski
\cite{BorodinCombin}, Section 1):
\begin{thm}\label{TheoremZmeasuresSpectralMeasures}
For any $n=1,2,\ldots$ and any $\lambda\in\Y_n$ we have
$$
M^{(n)}_{z,\bar z,\theta}(\lambda)=\frac{n!}{H(\lambda,\theta)}
\int\limits_{\omega=(\alpha,\beta)\in\Omega}\widetilde{P}_{\lambda}(\omega|\theta)M_{z,\bar z,\theta}^{\Spectral}(d\omega),
$$
where $\widetilde{P}_{\lambda}(\omega|\theta)$ denotes the image
of the Jack symmetric function $P_{\lambda}(x|\theta)$ (with the
parameter $\theta$)
 under the embedding defined by equation
(\ref{LambdaEmbedding}). Here $M_{z,\bar z,\theta}^{\Spectral}$ is an
unique probability measure on $\Omega$.
\end{thm}
\begin{rem} 1) For the definition of the Jack symmetric functions $P_{\lambda}(x|\theta)$
see Macdonald \cite{macdonald}, VI, \S 10. Note that our parameter
$\theta$ is inverse to Macdonald's $\alpha=1/\theta$.\\
2) The claim of  Theorem \ref{TheoremZmeasuresSpectralMeasures} is a consequence of a more general
statement proved in Kerov, Okounkov, Olshanski \cite{koo}.
\end{rem}
\subsection{A formula for the spherical function}
Let $\varphi_z$ be the spherical function of $T_{z,\frac{1}{2}}$.
In what follows we describe the expansion of $\varphi_z$ in terms
of the zonal spherical functions of the Gelfand pair $(S(2n),
H(n))$. The fact that $(S(2n),H(n))$ is a Gelfand pair implies
that there is an orthogonal basis $\{w^{\lambda}\}$ in
$C(S(2n),H(n))$ whose elements, $w^{\lambda}$, are the zonal
spherical functions of $(S(2n),H(n))$. The elements $w^{\lambda}$
are parameterized by Young diagrams with $n$ boxes, and are
defined by
$$
w^{\lambda}(g)=\frac{1}{|H(n)|}\sum\limits_{h\in
H(n)}\chi^{2\lambda}(gh),
$$
see Macdonald \cite{macdonald}, Sections VII.1 and VII.2. Here $|H(n)|$ is the number of elements in the hyperoctahedral group of degree $n$,
and
$\chi^{2\lambda}$ is the character of the irreducible
$S(2n)$-module corresponding to
$2\lambda=(2\lambda_1,2\lambda_2,\ldots )$.
\begin{prop}\label{RestrictionSphericalFunctions} Denote by $\varphi_z$ the spherical function of
$T_{z,\frac{1}{2}}$. We have
\begin{equation}\label{4.3-4.4}
\varphi_z|_{S(2n)}(g)=\sum\limits_{|\lambda|=n}M_{z,\bar z,
\theta=\frac{1}{2}}^{(n)}(\lambda)w^{\lambda}(g),\;\; g\in S(2n),
\end{equation}
where $M_{z,\bar z, \theta=\frac{1}{2}}^{(n)}(\lambda) $ is the
$z$-measure with the Jack parameter $\theta=1/2$.
\end{prop}
\begin{proof} See Strahov \cite{strahov1}, Proposition 6.3
\end{proof}
\subsection{Spectral representation}
Recall the definition of the coset type of a permutation from
$S(2n)$ (Macdonald \cite{macdonald}, VII, \S 2). If $g\in
S(2n)$, then we can associate with $g$ a graph $\Gamma(g)$
with vertices $1, 2,\ldots,2n$, and edges $\epsilon_i,
g\epsilon_i$ ($1\leq i\leq n$). Each edge $\epsilon_i$ joins the
vertices $2i-1$ and $2i$, and $g\epsilon_i$ joins the vertices
$g(2i-1)$ and $g(2i)$. Then the connected components of
$\Gamma(g)$ can be understood as cycles of even lengths
$2\rho_1,2\rho_2,\ldots,$ where $\rho_1\geq\rho_2\geq\ldots$. Thus
each $g\in S(2n)$ gives rise to a partition
$\rho=(\rho_1,\rho_2,\ldots )$ of $n$, called the coset type of
$g$. For example, if
$$
g=(135)(67)(248)\in S(8),
$$
then the graph $\Gamma(g)$ is shown on Fig.2, and the coset type
of $g$ is $\rho=(3,1)$.
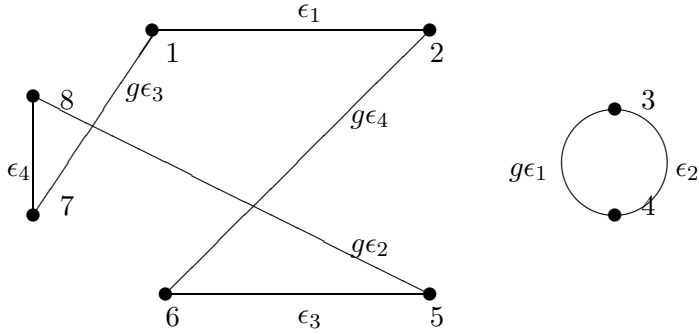
\begin{figure}
\begin{picture}(100,150)
\put(-55,140){\circle*{5}}
\put(-50,128){$1$}
 \put(-55,140){\line(1,0){105}}
\put(50,140){\circle*{5}}
\put(50,128){$2$}
\put(120,110){\circle*{5}}
\put(130,110){$3$}
\put(120,70){\circle*{5}}
\put(130,70){$4$}
\put(-100,115){\circle*{5}}
\put(-90,110){$8$}
\put(-90,70){$7$}
\put(-100,70){\line(2,3){47}}
\put(-100,70){\circle*{5}}
\put(-50,40){\circle*{5}}
\put(50,40){\circle*{5}}
 \put(-50,40){\line(1,0){100}}
\put(-50,28){$6$}
\put(50,28){$5$}
\put(120,90){\circle{40}}
 \put(-50,40){\line(1,1){100}}
 \put(50,40){\line(-2,1){150}}
 \put(-100,70){\line(0,1){45}}
\put(80,85){$g\epsilon_1$}
\put(143,85){$\epsilon_2$}
\put(-110,85){$\epsilon_4$}
\put(20,105){$g\epsilon_4$}
\put(0,145){$\epsilon_1$}
\put(0,29){$\epsilon_3$}
\put(20,56){$g\epsilon_2$}
\put(-65,115){$g\epsilon_3$}
\end{picture}
\caption{Graph $\Gamma(g)$, where $g=(135)(67)(248)\in S(8)$.}
\end{figure}
\begin{thm}\label{PropositionSpectralRepresentation}
Let $\varphi_z$ be the spherical function of the representation
$T_{z,\frac{1}{2}}$. There exists a unique probability measure
$M_{z,\theta=\frac{1}{2}}^{\Spectral}$ on $\Omega$ such that for
any $g\in S(2\infty)$
$$
\varphi_z(g)=\int\limits_{\Omega}\chi_{\theta=\frac{1}{2}}^{(w)}(g)M_{z,\theta=\frac{1}{2}}^{\Spectral}(d\omega),
$$
where
\begin{equation}\label{AnalogTheoremThomaFormula}
\chi_{\theta=\frac{1}{2}}^{(w)}(g)=
\prod\limits_{k=2}^{\infty}(\widetilde{p}_k(w|\theta=\frac{1}{2}))^{\rho_k(g)},
\end{equation}
and where
$\rho(g)=\left(\rho_1(g),\rho_2(g),\rho_3(g),\ldots\right)$
is the coset type of $g\in S(2\infty)$.
\end{thm}
\begin{rem}
a) The notion of coset-type for a permutation $g\in S(2n)$ can be
extended to elements of $S(2\infty)$ in an obvious way.\\
b) The formula in  Theorem \ref{PropositionSpectralRepresentation} can be understood as a
natural analogue of the spectral decomposition of the characters of
the generalized regular representation of the infinte symmetric
group, see Section 9.7 in Kerov, Olshanski,
and Vershik \cite{KOV1,KOV2}. The functions
$\chi_{\theta=\frac{1}{2}}^{(w)}(g)$ introduced in the statement
of  Theorem \ref{PropositionSpectralRepresentation} play the role of the extreme characters of
$S(\infty)$. Moreover, formula (\ref{AnalogTheoremThomaFormula})
is an analogue of an  explicit formula for the  extreme characters of
the infinite symmetric group given by the Thoma theorem.\footnote{For a modern presentation of the Thoma theorem see Kerov, Olshanski, and Vershik \cite{KOV2}, Section 9.}
\end{rem}
\begin{proof}We use  Theorem \ref{TheoremZmeasuresSpectralMeasures}, Proposition \ref{RestrictionSphericalFunctions}, and  find
$$
\varphi_z|_{S(2n)}(g)=
\int\limits_{\Omega}\left(\sum\limits_{|\lambda|=n}\frac{n!}{H(\lambda,\theta=\frac{1}{2})}
\widetilde{P}_{\lambda}(\omega|\theta=\frac{1}{2})w^{\lambda}(g)\right)M_{z,\bar z,\theta=\frac{1}{2}}^{\Spectral}(d\omega)
$$
Note that the zonal spherical functions $w^{\lambda}$ of the
Gelfand pair $(S(2n),H(n))$ are constant on the double cosets
$H(n)gH(n)$ of $H(n)$ in $S(2n)$. Two permutations
$g, g_1\in S(2n)$ have the same coset-type if and only if
$g_1\in H(n)gH(n)$. Therefore for any $g\in S(2n)$ we can
write
$$
w^{\lambda}(g)=w^{\lambda}_{\rho(g)},\;\rho(g)\;\mbox{is the
coset type of}\; g\in S(2n).
$$
Let $\widetilde{J}_{\lambda}(\omega|\theta=\frac{1}{2})$ be the image of the Jack symmetric function
$J_{\lambda}^{\alpha}$
with the parameter $\alpha=2$ under the embedding defined by equation
(\ref{LambdaEmbedding}).  The functions  $\widetilde{P}_{\lambda}(\omega|\theta=\frac{1}{2})$ are expressible in terms
of $\widetilde{J}_{\lambda}(\omega|\theta=\frac{1}{2})$, see Macdonald \cite{macdonald},VI, equation (10.22).
Taking this into account it is not hard to see that
the sum inside the integral above can be rewritten as
$$
\sum\limits_{|\lambda|=n}\frac{|H(n)|}{h(2\lambda)}w_{\rho(g)}^{\lambda}\widetilde{J}_{\lambda}(\omega|\theta=\frac{1}{2}),
$$
where  $h(2\lambda)$ is the product of the hook-lengths of the
partition $2\lambda$. Formula (2.16) of Macdonald
\cite{macdonald}, VII, \S 2 gives
$$
\sum\limits_{|\lambda|=n}\frac{|H(n)|}{h(2\lambda)}w_{\rho(g)}^{\lambda}\widetilde{J}_{\lambda}(\omega|\theta=\frac{1}{2})=\prod\limits_{k=2}^{\infty}
\left(\widetilde{p}_k(w|\theta=\frac{1}{2})\right)^{\rho_k(g)},
$$
and the statement of the Theorem follows.
\end{proof}
\section{The point process defined by $M_{z,\theta=\frac{1}{2}}^{\Spectral}$}
Now our aim is to describe the probability measures $M_{z,\theta=\frac{1}{2}}^{\Spectral}$. We use the idea proposed by Borodin and Olshanski
to view the infinite collection of parameters
$\left(\alpha_1,\alpha_2,\ldots;\beta_1,\beta_2\ldots\right)$ as random points distributed according to $M_{z,\theta=\frac{1}{2}}^{\Spectral}$.
In this way we obtain a point process, and the correlation functions of this point process provide a detailed description of $M_{z,\theta=\frac{1}{2}}^{\Spectral}$.

It is more convenient to work with lifted spectral measures $\widetilde{M}^{\Spectral}_{z,\theta=\frac{1}{2}}$ defined as follows.
Denote by $\widetilde{\Omega}$ the set of triples
$$\tilde{\omega}=(\alpha,\beta,\delta)\in\R^{2\infty}\times\R_{\geq 0},$$
where 
$$
\alpha=(\alpha_1\geq\alpha_2\geq\ldots\geq 0),\;
\beta=(\beta_1\geq\beta_2\geq\ldots\geq 0),\; \delta\in\R_{\geq
0},
$$ 
and 
$$
\sum\limits_{i=1}^{\infty}(\alpha_i+\beta_i)\leq\delta.
$$
Consider the map
$$
((\alpha,\beta),\delta)\in\Omega\times\R_{\geq 0}\rightarrow
(\delta\alpha,\delta\beta,\delta)\in\widetilde{\Omega}.
$$
By definition, the measure
$\widetilde{M}^{\Spectral}_{z,\theta=\frac{1}{2}}$ is the
pushforward of the measure
$$
M^{\Spectral}_{z,\theta=\frac{1}{2}}\otimes\left(\frac{s^{2z\bar z-1}}{\Gamma(2z\bar z)}e^{-s}ds\right)
$$
on $\Omega\times\R_{\geq 0}$ under this map. The procedure to
obtain the probability measure
$\widetilde{M}^{\Spectral}_{z,\theta=\frac{1}{2}}$ from the
probability measure $M^{\Spectral}_{z,\theta=\frac{1}{2}}$ is
called lifting.\footnote{This procedure was first introduced in Borodin \cite{Borodin-5} in the context of determinantal processes relevant for the harmonic analysis
on the infinite symmetric group. It leads to an essential simplification of the correlation functions.}
We define the embedding $\widetilde{\Omega}\rightarrow\Conf(\R\setminus\{0\})$ as 
$$
\tilde{\omega}=(\alpha,\beta,\delta)\rightarrow C=\left\{\delta\alpha_i\neq 0\right\}\cup\left\{-\delta\beta_j\neq 0\right\}.
$$
This way  we convert $\tilde{\omega}$ to a point configuration $C$ in $\R\setminus\{0\}$.
Given a probability measure on $\widetilde{\Omega}$, its pushforward under this embedding
is a probability measure on point configurations in $\R\setminus\{0\}$, i.e. a point process. In
particular, the probabilty measure $\widetilde{M}^{\Spectral}_{z,\theta=\frac{1}{2}}$ defines a point process in $\R\setminus\{0\}$.

Let $F:\R_{>0}^n\rightarrow\C$ be any continuous function with compact support. The equality
\begin{equation}
\begin{split}
\int\limits_{w=(\alpha,\beta)\in\Omega}\sum\limits_{i_1,\ldots,i_k}F(\alpha_{i_1},\ldots,\alpha_{i_n})\widetilde{M}^{\Spectral}_{z,\theta=\frac{1}{2}}(dw)
=\int\limits_{\R^n>0}F(x_1,\ldots,x_n)\varrho^{\Spectral}_n(dx)
\end{split}
\nonumber
\end{equation}
(where the sum is taken over pairwise distinct indexes) defines the correlation measures
$\varrho^{\Spectral}_n(dx)$. The correlation functions $\varrho^{\Spectral}_n(x)$ are densities of $\varrho^{\Spectral}_n(dx)$ with respect to the Lebesgue measure $dx$.
These correlation functions, $\varrho^{\Spectral}_n(x)$, describe correlations of $\{\alpha_i\}$. It is possible to consider
joint correlations of $\{\alpha_i\}$ and $\{\beta_i\}$, but in this paper we deal with correlations of $\{\alpha_i\}$ only.

In order to present the main result of the paper let us introduce
the following notation. First, we introduce the functions
$w_a(x;z,z')$ indexed by $a\in\Z+\frac{1}{2}$, parameterized by two complex parameters $z$ and $z'$, and whose argument $x$ varies
in $\R_{>0}$. They are expressed through the classical Whittaker
functions $W_{k,m}(x)$, see Andrews, Askey, and Roy
\cite{andrews}, Section 4, for a definition. The functions
$w_a(x;z,z')$ are defined in terms of   $W_{k,m}(x)$ as 
\begin{equation}
w_a(x;z,z')=\left(\Gamma(z-a+\frac{1}{2})\Gamma(z'-a+\frac{1}{2})\right)^{-\frac{1}{2}}x^{-\frac{1}{2}}W_{\frac{z+z'}{2}-a,\frac{z-z'}{2}}(x).
\end{equation}
Since $W_{k,\mu}(x)=W_{k,-\mu}(x)$, this expression is symmetric
with respect to $z\longleftrightarrow z'$.

Second, we introduce a function of two arguments called the
(scalar) Whittaker kernel. This function is parameterized by two
complex parameters, $z$ and $z'$, and can be  written
as

\begin{equation}
 K^{W}_{z,z'}(x,y)
=\sqrt{zz'}\frac{w_{-\frac{1}{2}}(x;z,z')w_{\frac{1}{2}}(y;z,z')-w_{\frac{1}{2}}(x;z,z')w_{-\frac{1}{2}}(y;z,z')}{x-y}.
\end{equation}

\begin{thm}\label{THEOREMMAINRESULT} For any $z\in\C\setminus\{0\}$
\begin{equation}\label{varrhoequationr}
\varrho_n^{\Spectral}(x_1,\ldots,x_n)=\Pf\left[\mathbb{K}^{\Spectral}(x_i,x_j)\right]_{i,j=1}^{n}.
\end{equation}
 The $2\times 2$ matrix valued correlation kernel  $\mathbb{K}^{\Spectral}(x,y)$ can be written as
$$
\mathbb{K}^{\Spectral}(x,y)=
\left[\begin{array}{cc}
  S(x,y) & \frac{d}{dy}S(x,y) \\
  \frac{d}{dx}S(x,y) & \frac{d^2}{dxdy}S(x,y)\\
\end{array}\right],
$$
where
\begin{equation}
\begin{split}
S(x,y)&=-\frac{1}{2}\sqrt{y}\int\limits_{x}^{+\infty}K^{W}_{-2z,-2\bar{z}}(s,y)\frac{ds}{\sqrt{s}}\\
&+\frac{\sqrt{z\bar{z}}}{4}\left(\int\limits_{x}^{+\infty}w_{-\frac{1}{2}}(s;-2z,-2\bar{z})\frac{ds}{\sqrt{s}}\right)
\left(\int\limits_{y}^{+\infty}w_{\frac{1}{2}}(t;-2z,-2\bar{z})\frac{dt}{\sqrt{t}}\right).
\end{split}
\nonumber
\end{equation}
\end{thm}
\begin{rem}
a) All matrix elements of  $\mathbb{K}^{\Spectral}(x,y)$ are constructed in terms of the Whittaker functions.\\
b) From the definition of the Whittaker functions it follows   that $\mathbb{K}^{\Spectral}(x,y)$ is real valued.\\
c) It can be shown that
for any $z\in\C\setminus\left\{0\right\}$ the following condition is satisfied
$$
S(x,y)=-S(y,x).
$$
Thus the matrix inside the Pfaffian in equation (\ref{varrhoequationr}) is antisymmetric.\\
d) The function $K^{W}_{z,\bar{z}}(x,y)$ is the (scalar) correlation kernel of a determinantal process which arises in the study 
of the decomposition of the generalized regular representation of $S(\infty)$ into irreducible components, see Borodin and Olshanski \cite{Borodin-3}.
\end{rem}
\section{Idea of the proof of Theorem \ref{THEOREMMAINRESULT}}

The first step in the proof of Theorem \ref{THEOREMMAINRESULT} is to construct a point process on
the lattice $\Z+\frac{1}{2}$, which converges to the point process defined by $\widetilde{M}^{\Spectral}_{z,\theta=\frac{1}{2}}$.
Given a box $b=(i,j)$ of a Young
diagram and a parameter $\theta>0$ the number
$$
c_{\theta}(b)=(j-1)-\theta(i-1)
$$
is referred to as $\theta$-content of the box $b$. A box $b$ of a Young diagram is
said to be positive or negative according to the sign of its
$\theta$-content. We can consider any Young diagram
$\lambda=(\lambda_1,\ldots,\lambda_l)$  as a
collection of boxes
$$
\lambda\equiv\left\{b(i,j): 1\leq i\leq l, 1\leq j\leq
\lambda_i\right\},
$$
and we  can split $\lambda$ into a union of disjoint subsets of its
positive and negative boxes
$$
\lambda^+=\left\{b\in\lambda|c_{\theta}(b)>0\right\},\;\lambda^{-}=\left\{b\in\lambda|c_{\theta}(b)\leq
0\right\}.
$$
Denote by $r$ the number of rows in $\lambda^+$, and by $s$ the
number of columns in $\lambda^{-}$. Let
$$
a_1\geq a_2\geq\ldots\geq a_r>0,\;\; b_1\geq b_2\geq\ldots\geq
b_s>0
$$
denote the lengths of corresponding rows and columns. Then the
expression
$$
\lambda=\left(a_1,\ldots,a_r\biggl|b_1,\ldots,b_s\right)
$$
defines coordinates of the Young diagram $\lambda$, and these
coordinates are dependent on the Jack parameter $\theta$. Now each Yound diagram $\lambda$ can be represented as a point configuration on
$\Z+\frac{1}{2}$ using the map
$$
\lambda\longrightarrow(A|B)_{\theta}(\lambda)=(-a_1-\frac{1}{2},\ldots,-a_r-\frac{1}{2};b_1+\frac{1}{2},\ldots,b_s+\frac{1}{2}).
$$
Next define a probability measure $\widetilde{M}_{z,\bar z,\theta,\xi}$ on the set of all Young diagrams $\Y$
by
\begin{equation}
\widetilde{M}_{z,\bar z,\theta,\xi}=(1-\frac{z\bar z}{\theta})^{\frac{z\bar z}{\theta}}\frac{(\frac{z\bar z}{\theta})_n}{n!}\xi^nM_{z,\bar z,\theta}^{(n)}(\lambda),\;\;|\lambda|=n.
\end{equation}
Here $0<\xi<1$ is an additional parameter. We refer to $\widetilde{M}_{z,\bar z,\theta,\xi}$ as to the mixed $z$-measures with an arbitrary Jack parameter $\theta>0$.
Define the lattice correlation functions $\varrho_{n,\theta,\xi}$ with the general parameter $\theta>0$
as  probabilities that the set $(A|B)_{\theta}(\lambda)$ contains a fixed subset $X$ of $\Z_{\geq 0}+\frac{1}{2}$. More precisely,
$$
\varrho_{n,\theta,\xi}(X)=\widetilde{M}_{z,\bar z,\theta,\xi}\left(\lambda\vert X\subset(A|B)_{\theta}(\lambda)\right),\;\;X\subset\Z_{\geq 0}+\frac{1}{2}.
$$
The crucial fact which enables us to compute the correlation functions of  $\widetilde{M}^{\Spectral}_{z,\theta=\frac{1}{2}}$, and  to prove
Theorem \ref{THEOREMMAINRESULT} is the following
\begin{prop}\label{PropLimitRelationSpectralMK}
Let $\xi\nearrow 1$, and assume that
 $x_1,\ldots,x_n\rightarrow +\infty$ inside
$\Z_{\geq 0}+\frac{1}{2}$ such that
\begin{equation}
\begin{split}
(1-\xi)x_1\rightarrow u_1,\ldots,(1-\xi)x_n\rightarrow u_n,
\end{split}
\nonumber
\end{equation}
where  $u_1,\ldots,u_n$ are pairwise distinct points of $\R_{>0}$. \\
Then  we have
\begin{equation}\label{LimitRelationSpectralMK}
\begin{split}
\varrho^{\Spectral}_n(u_1,\ldots,u_n)=\underset{\xi\nearrow
1}{\lim}\;(1-\xi)^{-n}\varrho_{n,\theta=\frac{1}{2},\xi}\left(x_1,\ldots,x_n\right).
\end{split}
\nonumber
\end{equation}
\end{prop}
(This statement can be proved by repetition of arguments from Borodin and Olshanski \cite{BorodinCombin}).
Proposition \ref{PropLimitRelationSpectralMK} shows that in order to compute the correlation function $\varrho_n$, it is necessarily
to have a formula for the correlation function $\varrho_{n,\theta=\frac{1}{2},\xi}$
defined by the mixed $z$-measure  with the Jack parameter $\theta=\frac{1}{2}$.
By Proposition \ref{PropositionMSymmetries} it is enough to compute the correlation functions $\varrho_{n,\theta=2,\xi}$
 defined by the mixed $z$-measure with the Jack parameter parameter $\theta=2$.
Such correlation functions were given in Strahov \cite{strahov2}. It was shown (see Strahov \cite{strahov2}, Section 2) that the correlation functions $\varrho_{n,\theta=2,\xi}$
can be written as Pfaffians with $2\times 2$ matrix valued kernels. Moreover,  these kernels are expressible in terms of the  Gauss hypergeometric functions, and  for matrix elements of these kernels there are double contour integral representations (see
Strahov \cite{strahov2}, Proposition 2.9, Proposition 2.10, equation (2.6)). Using these results  it is possible to compute the scaling limit of
the correlation functions $\varrho_{n,\theta=2,\xi}$ as $\xi\nearrow 1$, and to obtain the formula for $\varrho^{\Spectral}_n$ in  Theorem \ref{THEOREMMAINRESULT}.

\end{document}